\documentclass[10pt]{amsart}
\usepackage{amsmath, amsthm, amscd, amsfonts}

\makeatletter \oddsidemargin.9375in \evensidemargin \oddsidemargin
\newtheorem{theorem}{Theorem}[section]
\newtheorem{lemma}[theorem]{Lemma}
\newtheorem{proposition}[theorem]{Proposition}
\newtheorem{corollary}[theorem]{Corollary}
\theoremstyle{definition}
 \newtheorem{definition}[theorem]{Definition}

\theoremstyle{remark}
\newtheorem{remark}[theorem]{Remark}
\numberwithin{equation}{section}

\begin{document}

\title[ Stability of Fredholm property]{stability of Fredholm property for regular operators on Hilbert $C^*$-modules}

\author[ Marzieh Forough]{M. Forough}

\address{ School of Mathematics, Institute for Research in Fundamental Sciences (IPM), P.O. Box 19395-5746, Tehran, Iran}
\email{mforough@ipm.ir, mforough86@gamil.com}

%\address{}

%\email{mforough86@gmail.com}



\begin{abstract}
We study the stability of  Fredholm property for regular operators on Hilbert $C^*$-modules under some certain perturbations. We treat this problem when perturbing operators are (relatively) bounded or relatively compact.  We also consider the perturbations of regular Fredholm operators in terms of the gap metric. In particular, we prove that the space of all regular Fredholm operators on a Hilbert $C^*$-module $E$ is open in the space of all regular operators on $E$ with respect to the gap metric. As an application, we construct some continuous paths of selfadjoint regular Fredholm operators with respect to the gap metric. \\
%Also, a charecterization of finitely generated Hilbert $C^*$-modules over unital $C^*$-algebras is obtained.

\end{abstract} \maketitle

\section{Introduction}

In this paper, we investigate some forms of perturbations of regular Fredholm operators on Hilbert $C^*$-modules. 

As in the Hilbert  space case, one needs to consider unbounded adjointable operators on Hilbert $C^*$-modules, see \cite{Baaj}, \cite{Wor}, \cite{Kuc}, etc.  The lack of the projection theorem in Hilbert $C^*$-modules caused the theory of unbounded operators on Hilbert $C^*$-modules to be more complicated than Hilbert space case. The regularity condition was assumed for a closed densely defined operator on a Hilbert  $C^*$-module $E$ with densely defined adjoint, and this leads to have more reasonable theory like having a continuous functional calculus. 
\begin{definition}
Let $t: D(t)\subseteq E\rightarrow E$ be a closed densely defined operator with densely defined adjoint on a Hilbert $C^*$-module $E$. Then $t$ is called regular if $I+t^{*}t$ has dense range.
\end{definition}
In this paper, we denote by $R(E)$ the set of all regular operators on a Hilbert $C^*$-module $E$. We say that a regular operator $t$ is selfadjoint if $t=t^{*}$, and we denote the set of all selfadjoint regular operators on $E$ by $SR(E)$.

 Unbounded Fredholm operators on Hilbert $C^*$-modules have been studied in several papers, see \cite{Joachim}, \cite{Wahl1}, among others. Joachim in \cite{Joachim} argued that it is natural to consider unbounded Fredholm operators to describe some important invariants taking values in $K^{0}(X; A)$ and $K^{1}(X; A)$, where $X$ is a compact space and $A$ is a unital $C^*$-algebra. In \cite{Wahl1}, the notion of noncommutative spectral flow for paths of  selfadjoint regular Fredholm operators is defined and studied.
 \begin{definition} (c.f \cite{Joachim})
Let $t: D(t)\subseteq E\rightarrow E$ be a regular operator. A bounded adjointable operator $G \in L(E)$ is called a pseudo left inverse of $t$ if $Gt$ is closable and its closure $\overline{Gt}\in L(E)$ and $\overline{Gt}\equiv1$ mod$K(E)$. 

Analogously, $G \in L(E)$ is called a pseudo right inverse of $t$ if $tG$ is closable and its closure $\overline{tG}\in L(E)$ and $ \overline{tG}\equiv 1$ mod$K(E)$.

 A regular operator $t$ on a Hilbert $C^*$-module $E$ is called Fredholm if $t$ has a pseudo right inverse as well as a pseudo left inverse.
\end{definition}

The stability of Fredolm property for linear operators between Banach spaces or Hilbert spaces are studied extensively, see  \cite{G-K}, \cite{Kato1}, and  \cite{Nagy}. One basic problem here is how to define a small perturbation for unbounded operators. One rather general definition is based on the notion of relatively bounded perturbations, however, it is still restrictive, see \cite{Kato}. The most natural and general definition of smallness of a perturbation is given by the gap metric in the space of closed linear operators between Banach spaces. In \cite{Kato}, a systematic use of the gap metric was made. In this paper, we aim to study the preservation of Fredholm property for regular operators  on Hilbert $C^*$-modules. Following the Banach space case, we consider the small perturbations by notions of (relatively) bounded perturbations, (relatively) compact perturbations and the gap metric on the space of regular operators. 
\begin{theorem} \label{main-a}
 $(a)$ Let $t: D(t)\subseteq E\rightarrow E$ be a regular Fredholm operator on a Hilbert $C^*$-module $E$. Then there exists  $\varepsilon > 0$ with the following property:

$(i)$  $D+t$ is a regular Fredholm operator for any selfadjoint bounded operator $D$ on $E$ satisfying $\|D\|\leq \varepsilon$. \\

$(ii)$ If $s: D(s)\subseteq E\rightarrow E$ is a selfadjoint regular operator on $E$ satisfying $D(t) \subseteq D(s)$ and 
$$\|s(x)\|\leq \varepsilon (\|(t+s)(x)\|+\|x\|)$$
for all $x \in D(t)$. Then $t+s$ is a selfadjoint regular Fredholm operator.\\
$(iii)$ If $s$ is a regular operator on $E$ with $g(s,t)\leq \varepsilon$ then $s$ is Fredholm, where $g$ is the gap metric.\\
$(b)$  Let $t$ and $s$ be  selfadjoint regular operators on a countably generated Hilbert $C^*$-module $E$ over unital $C^*$-algebra $A$. Suppose that $t$ is Fredholm and $s$ is $t$-compact, then $s+t$ is a selfadjoint regular Fredholm operator. \\\end{theorem}
Studying the perturbations of unbounded Fredholm operators on Banach spaces or Hilbert spaces makes a heavy use of the fact that such operators have closed range, finite dimensional kernel and cokernel.  However, a regular Fredholm operator does not have necessarily a closed range or finite dimensional kernel and  cokernel.  Hence we can not employ the ideas and techniques applied in studying the perturbations of unbounded closed densely defined Fredholm operators on Hilbert spaces in our problem concerning the perturbations of regular Fredholm operators on Hilbert $C^*$-modules. To tackle our problem, we employ the Cayley transforms for  selfadjoint regular operators, and obtain a necessary and sufficient condition for a selfadjoint regular operator to be Fredholm in terms of its Cayley transform. 

\begin{theorem} 
Let $t$ be a selfadjoint regular operator on a Hilbert $C^*$-module $E$. Then $t$ is Fredholm if and only if bounded adjointable operator $I+C_{t}$ is Fredholm.
\end{theorem}

Then we determine an explicit relation between the Cayley transform of a selfadjoint regular operator and the Cayley transform of its perturbed operator. This enables us to translate our problem in unbounded case to a problem in bounded adjointable case and use the results available in the stability of Fredholm property for bounded adjointable operators.

 In the following proposition, we collect some results from \cite{Exel} on the stability of Fredholm property for bounded adjointable operators on Hilbert $C^*$-modules. 
 
\begin{proposition}\label{basic}(\cite{Exel})
 Let $A$ be a bounded adjointable Fredholm operator on a Hilbert $C^{*}$-module $E$, then the followings hold.\\
$(i)$ There is  $\varepsilon > 0$ such that any bounded adjointable operator $D$ satisfying $\|A-D\|\leq 
\varepsilon $ is Fredholm.\\
$(ii)$ If $D$ is a bounded adjointable operator on $E$ such that $D-A$ is a compact operator on $E$, then $D$ is Fredholm.\\
$(iii)$ If bounded adjointable operators $U$ and $V$ on $E$ are invertible, then $UAV$ is Fredholm.
\end{proposition}

%In this paper, we denote by $R(E)$, $L(E)$, and $K(E)$ the spaces of regular operators, bounded adjointable operators, and compacts operators on $E$, respectively.
%Suppose that $t$ is a regular operator on $E$. Let define  $Q_{t}=(I+t^{*}t)^{\frac{-1}{2}}$ , $R_{t}=(I+t^{*}t)^{-1}$ and $F_{t}=tQ_{t}$, then  $0\leq Q_{t}\leq 1$ in $L(E)$ and $F_{t}\in L(E)$ (see chapter 9 of \cite{Lance}). We call $F_{t}$ the bounded transform of  regular operator $t$.\\
 
This paper is organized as follows. In section 2, we give a criterion for a selfadjoint regular operator on a Hilbert $C^*$-module to be Fredholm in terms of its Cayley transform $C_{t}$. We also present a relation between the Cayley transform of a selfadjoint regular operator and the Cayley transform of its perturbation by a selfadjoint bounded operator or a selfadjoint regular operator.\\
In section 3, first we examine the stability of Fredholm property for regular operators when perturbing operators are bounded adjointable operators. Then we turn to the case of unbounded perturbing operators. We consider the relatively bounded perturbations and relatively compact perturbations. Finally, we investigate the stability of  Fredholm property for regular operators in terms of the gap metric. In particular, we show that the space of regular Fredholm operators is open in $R(E)$ with respect to the gap metric.

\section{ Cayley transforms of selfadjoint regular  Fredholm operators}

In this section we deal with the Cayley transform of a selfadjoint regular Fredholm operator.  

Let us start with  recalling some notations regarding regular operators. Let $t: D(t)\subseteq E\rightarrow E$ be a regular operator on a Hilbert $C^*$-module $E$, then $t^*$ and $tt^*$ are regular operators. Define $Q_{t}=(1+t^{*}t)^{\frac{-1}{2}}$, $R_{t}=Q_{t}^{2}$ and $F_{t}=tQ_{t}$, then $Q_{t}$ and $F_{t}$ are bounded adjointable operators and $D(t)=ran(Q_{t})$ (c.f \cite{Lance}). The operator $F_{t}$ is called the bounded transform of $t$. We recall from \cite{Joachim} that a regular operator $t$ is Fredholm if and only if its bounded transform $F_{t}$ is Fredholm.

Suppose that $t$ is a selfadjoint regular operator then by Lemma 9.7 and Lemma 9.8 of \cite{Lance}, $t+i$ and $t-i$ are bijections. Define the Cayley transform of $t$, denoted by $C_{t}$, by 
$C_{t}:= (t-i)(t+i)^{-1}$. Then $C_{t}$ is a selfadjoint bounded operator.
Furthermore,  by Theorem 10.4 of \cite{Lance}, the Cayley transform defines a bijection from $SR(E)$ onto 
$$\{U \in L(E): U \ \ is \ \ unitary; I-U \ \ has \ \ dense \ \ range\}.$$
In the next  theorem, we give a criterion for a selfadjoint regular operator to be Fredholm in terms of its Cayley transform. This characterization is essential to obtain our main results.

\begin{theorem}\label{main1}
 Let $t$ be a selfadjoint regular operator on a Hilbert $C^*$-module $E$. Then $t$ is Fredholm if and only if $I+C_{t}$ is Fredholm.
 \end{theorem}
\begin{proof}
 Thanks to Lemma 2.2 of \cite{Joachim}, it suffices to prove that $F_{t}$ is Fredholm if and only if $I+C_{t}$ is Fredholm. To see this, first note that

$$(t-i)^{-1}=(t+i)(t^{2}+1)^{-1}=tR_{t}+iR_{t},$$
$$ (t+i)^{-1}=(t-i)(t^{2}+1)^{-1}=tR_{t}-iR_{t}.$$
 These identities imply that
$$2tR_{t}=(t-i)^{-1}+(t+i)^{-1}.$$
Using above identity together with  $Q_{t}tQ_{t}=tR_{t}$ ((10.9) of \cite{Lance}), we can get

$$I+C_{t}=I+(t-i)(t+i)^{-1}=$$
$$(t-i)((t-i)^{-1}+(t+i)^{-1})=$$
$$2(t-i)tR_{t}=2t^{2}R_{t}-2itR_{t}$$
$$2tQ_{t}tQ_{t}-2itQ_{t}Q_{t}=$$
\begin{equation}\label{a}
2F_{t}(F_{t}-iQ_{t}).
\end{equation}\label{main11}
Observe that (10.3) and (10.9) of \cite{Lance}
shows that $F_{t}-iQ_{t}$ is isometry. Moreover, $ran(F_{t}-iQ_{t})=ran(t-i)$ since $ran(Q_{t})=D(t)$. Now, Lemma 9.8 of \cite{Lance} yields that $F_{t}-iQ_{t}$ is surjective, and so by Theorem 3.5 of \cite{Lance}, $F_{t}-iQ_{t}$ is unitary. Therefore, by Proposition \eqref{basic}, identity $I+C_{t}= 2F_{t}(F_{t}-iQ_{t})$ shows that $F_{t}$ is Fredholm if and only if $I+C_{t}$ is Fredholm.
\end{proof}
We remark here that a similar criterion for a selfadjoint closed densely defined operators on a Hilbert space to be Fredholm is obtained in \cite{Phillips}. However, the proof given in \cite{Phillips} is different with our proof for the case of selfadjoint regular operators on Hilbert $C^*$-modules.

 As an application of Theorem \ref{main1}, we can obtain a sufficient condition for a selfadjoint regular operator to be Fredholm. This condition is well-known among experts, however we give other proof based on Theorem \ref{main1}.
 
\begin{corollary}\label{resolvent}
Assume that $t$ is a selfadjoint regular operator on a Hilbert $C^*$-module $E$ such that $(t+i)^{-1}$ is compact. Then $t$ is Fredholm. 
\end{corollary}
\begin{proof}
 Observe that 
$$I+C_{t}=2I-(I-C_{t})=2I-( I-(t-i)(t+i)^{-1})$$
$$2I-((t+i)-(t-i))(t+i)^{-1}$$
$$=2I-2i(t+i)^{-1}.$$
  Thus $I+C_{t}$ is Fredholm since it is a compact perturbation of Fredholm operator $2I$. Now, employ Theorem \ref{main1} to conclude that $t$ is Fredholm.
\end{proof}

 \begin{remark}
 Theorem \ref{main1} together with Theorem 10.4 of \cite{Lance} imply that the Cayley transform maps selfadjoint regular Fredholm operators on a Hilbert $C^*$-module $E$ onto the set
 $$\{U \in L(E): U \ \ is \ \ unitary \ \ ; I-U \ \ has \ \ a\ \ dense \ \ range; \ \  and \ \  I+U \ \ is\ \ Fredholm \}.$$
\end{remark}
 To study the bounded perturbations or relatively bounded perturbations of selfadjoint regular Fredholm operators, we need to determine the relation between the Cayley transform of the perturbed operator and the Cayley transform of the original operator.
\begin{proposition} \label{main2}
 Let $t$ be a selfadjoint regular operator acting on a Hilbert $C^*$-module  $E$ and $D$ is a selfadjoint bounded operator on $E$. Then we have
$$ C_{t+D}-C_{t}=(I-C_{t})D(D+t+i)^{-1}.$$
\end{proposition}
\begin{proof}
The fact that operators $t-i$ and $t+i$ are bijections allows us to compute as follows:
 $$C_{D+t}-C_{t}=(D+t-i)(D+t+i)^{-1}-(t-i)(t+i)^{-1}=$$
$$(t-i)((D+t+i)^{-1}-(t+i)^{-1})+D(D+t+i)^{-1}=$$
$$(t-i)(t+i)^{-1}(-D)(t+D+i)^{-1}+D(D+t+i)^{-1}=$$
$$(I-C_{t})D(D+t+i)^{-1}.$$
\end{proof}

The following proposition can be deduced by a similar argument given in above proposition.
\begin{proposition} \label{main3}
 Let $t$ be a selfadjoint regular operator on a Hilbert $C^*$-module  $E$ and $s$ be a selfadjoint regular operator on $E$. Suppose  that $t+s$ is a selfadjoint regular operator then
$$ C_{t+s}-C_{t}=(I-C_{t})s(s+t+i)^{-1}$$
\end{proposition}
We recall here that there have been obtained some sufficient conditions which  guarantee that sum of two selfadjoint regular operators remains regular and selfadjoint. In \cite{van}, it is discussed that the perturbations of a regular selfadjoint operator by a local bounded selfadjoint regular operator is again regular and selfadjoint. Moreover, in Theorem 4.5 of \cite{Kaad-L}, Kato-Rellich theorem on Hilbert spaces is extended to Hilbert $C^*$-modules, we use this theorem in the next section. 

\section{ Perturbations of  regular Fredholm operators}
In this section, we study the stability of Fredholm property for regular operators under certain perturbations.
% First, we discuss cases in which perturbing operators are bounded adjointable.
\subsection{Bounded perturbations}
We start this subsection with the following observation which can be useful in some of our proofs.
\begin{remark}\label{norm}
Let $t$ be a selfadjoint regular operator acting on a Hilbert $C^*$-module $E$. Then we have 
$$(t-i)^{-1}=(t+i)(t^{2}+1)^{-1}=tR_{t}+iR_{t}=F_{t}Q_{t}-iQ_{t}^{2}.$$
 Hence $\|(t-i)^{-1}\| \leq 2$.
\end{remark}

 \begin{theorem} \label{bd1}
Let $t$ be a selfadjoint regular Fredholm operator on a Hilbert $C^*$-module $E$. Then there exists  $\varepsilon > 0$ such that $D+t$ is a selfadjoint regular Fredholm operator for any selfadjoint bounded operator $D$ on $E$ satisfying $\|D\|\leq \varepsilon$.  
\end{theorem}
\begin{proof}
 By Theorem \ref{main1}, $I+C_{t}$ is Fredholm. Apply Proposition \ref{basic} to get  $\varepsilon > 0$ with the following property: If $S$ is a  bounded adjointable operator on $E$ satisfying $\|(I+C_{t})-S\|\leq 4 \varepsilon$, then $S$ is Fredholm.  Assume that $D$ is a selfadjoint bounded operator on $E$ with $\|D\|\leq \varepsilon$.  Then it follows by Proposition \ref{main2} and Remark \ref{norm} that
$$\|(I+C_{t+D})-(I+C_{t})\|\leq \|I-C_{t}\|\|D\|\|(D+t+i)^{-1}\|\leq 4 \varepsilon. $$
Hence $I+C_{D+t}$ is Fredholm and so is $D+t$, as desired.
\end{proof}

\begin{remark}\label{rm}
 We recall from \cite{Gap} that one can correspond to any regular operator $t$ on Hilbert $C^*$-module $E$ a selfadjoint regular operator $\tilde{t}=\begin{pmatrix} 0&t^{*} \cr t&0 \end{pmatrix}$ on $E\oplus E$.
\end{remark}
\begin{lemma}\label{sadj}
Let $t$ be a regular Fredholm operator on a Hilbert $C^*$-module $E$, then $\tilde{t}=\begin{pmatrix} 0&t^{*} \cr t&0 \end{pmatrix}$ is a selfadjoint regular Fredholm operator.
\end{lemma}
\begin{proof}
Note that $F_{\tilde{t}}=\begin{pmatrix} 0&F_{t^{*}} \cr F_{t}&0 \end{pmatrix}=\begin{pmatrix} 0&F_{t}^{*} \cr F_{t}&0\end{pmatrix}$, since $F_{t^{*}}=F_{t}^{*}$.  Now, it is clear that $F_{\tilde{t}}$ is Fredholm if $F_{t}$ is Fredholm.
\end{proof}
Remark \ref{rm} enables us to omit the assumption of being  selfadjoint in Theorem \ref{bd1}.
\begin{proposition} \label{bd2}
Let $t$ be a regular operator on a Hilbert $C^*$-module $E$. Then there exists $\varepsilon > 0$ such that $D+t$ is regular Fredholm operator on $E$ for any bounded adjointable operator $D$ on $E$ satisfying $\|D\|\leq \varepsilon$.
\end{proposition}
\begin{proof}
 If  $t$ is a regular operator on $E$ and $D$ is a bounded adjointable operator on $E$ then clearly we have $(D+t)^{*}=D^{*}+t^{*}$ and so  $\widetilde{(D+t)}=\tilde{D}+\tilde{t}$. Also, it is easy to see that $\|\tilde{D}\|=\|D\|$.  Now apply Lemma \ref{sadj} and Theorem \ref{bd1} to conclude proposition.
\end{proof}
\begin{proposition}
 Suppose $t$ is a regular Fredholm operator on a Hilbert $C^*$-module $E$ and $K$ is a compact operator on $E$. Then regular operator $t+K$ is Fredholm.
\end{proposition}
\begin{proof}
By Remark \ref{rm}, it suffices to prove that if $t$ is a selfadjoint regular Fredholm operator and $K$ is a selfadjoint compact operator then selfadjoint regular operator $t+K$ is Fredholm. Indeed, compactness of operator $K$ together with Proposition \eqref{main2} imply that $I+C_{t+K}$ is a compact perturbation of Fredholm operator $I+C_{t}$. Hence Proposition \ref{basic} implies that $I+C_{t+K}$ is Fredholm and so $t+K$ is Fredholm.  
\end{proof}
\subsection{Unbounded perturbations}

\begin{theorem}
 Let $t: D(t) \subseteq \rightarrow E$ be a selfadjoint  regular Fredholm operator on a Hilbert $C^*$-module $E$. Then there exists  $\varepsilon > 0$ with the following property: If $s: D(s) \subseteq \rightarrow E$ is a selfadjoint regular operator satisfying $D(t) \subseteq D(s)$ and 
$$\|s(x)\|\leq \varepsilon (\|(t+s)(x)\|+\|x\|)$$
for all $x \in D(t)$. Then $t+s$ is a selfadjoint regular Fredholm operator.

\end{theorem}
\begin{proof}
 Note that Theorem \ref{main1} and Proposition \ref{basic} yield that there is a positive number $\varepsilon^{\prime}$ such that any bounded adjointable operator $D$ on $E$ satisfying $\|(I+C_{t})-D\|\leq \varepsilon^{\prime}$ is also Fredholm. Put $\varepsilon=min\{\frac{1}{4},\frac{\varepsilon^{\prime}}{10}\}$. Suppose that $s$ is a selfadjoint regular operator on $E$ with $D(t) \subseteq D(s)$ and satisfying
\begin{equation}\label{unbd}
 \|s(x)\|\leq \varepsilon(\|(t+s)(x)\|+\|x\|)
\end{equation}
for all $x \in D(t)$. We are going to show that $t+s$ is a selfadjoint regular Fredholm operator. For this, first note that it follows from  $\varepsilon \leq \frac{1}{4}$ and \ref{unbd} that
\begin{equation}
\|s(x)\|\leq \frac{1}{3}(\|x\|+\|t(x)\|)
\end{equation}\label{1}
for all $x \in D(t)$. Hence Theorem 4.4 of \cite{Kaad-L} implies that $t+s$ is a selfadjoint regular operator on $E$. Observe that $(t+s+i)^{-1}(x) \in D(t)$ for all $ x\in E$, so the inequality (4.1) yields that
$$\|s(s+t+i)^{-1}(x)\| \leq \varepsilon(\|(s+t)(s+t+i)^{-1}x\|+\|(s+t+i)^{-1}(x)\|)$$
$$\leq \varepsilon\|(s+t+i)(s+t+i)^{-1}x\|+2 \varepsilon\|(s+t+i)^{-1}(x)\|)$$
$$\leq 5\varepsilon\|x\|$$
for all $x \in E$. Employ Proposition \ref{main3} to conclude 
\begin{equation}
\|(I+C_{t+s})(x)-(I+C_{t})(x)\|\leq \|I-C_{t}\|\|s(s+t+i)^{-1}(x)\| 
\leq 10\varepsilon\|x\| \leq \varepsilon^{\prime}\|x\|
\end{equation}
 for all $x\in E$. Thus $I+C_{t+s}$ is Fredholm, and so $t+s$ is Fredholm, as desired.
\end{proof}
%Lets recall definition of $t-$compact from ?

%Indeed this definition is the adoption of the notion of relatively compact operators in the Hilbert spaces (see \cite{Kato}).

For a closed operator on a Hilbert space, the domain equipped with the graph scaler product is itself Hilbert space, see \cite{Kato}. In \cite{Kaad-L}, the analogous construction is discussed for a (semi)regular operator. Let $t$ be a regular operator and $x, y \in D(t)$, put
$$<x,y>_{t}=<x,y>+<t(x),t(y)>$$
 Then $D(t)$  with this inner product becomes a Hilbert $C^*$-module, denoted by $E(t)$. It follows from Proposition 2.4 of \cite{Kaad-L} that the inclusion map $ \iota: D(t) \rightarrow E$ is a bounded adjointable operator.
 
 \begin{definition}
 Let $t: D(t) \subseteq \rightarrow E$ and $s: D(s) \subseteq \rightarrow E$ be regular operators on a Hilbert $C^*$-module $E$ such that $D(t) \subseteq D(s)$.  We call $s$ is relatively $t$-compact or simply, a $t$-compact  if $s: E(t)\rightarrow E$ is a compact operator.
\end{definition}

\begin{theorem}\label{compact}
 Let $t$ and $ s$ be selfadjoint regular operators on a countably generated Hilbert $C^*$-module $E$ over unital $C^*$-algebra $A$. Suppose that $t$ is Fredholm and $s$ is a $t$-compact operator, then $s+t$ is a selfadjoint regular Fredholm operator. 
\end{theorem}
\begin{proof}
First note that  Lemma A.4 of \cite{Wahl1} yields that  $s+t$ is a selfadjoint regular operator. It is straightforward to check that $(t+i)^{-1}: E \rightarrow E(t)$ is bounded adjointable operator.  Hence  $s(t+i)^{-1}$ is a compact operator on $E$ since we assume that $s$ is  $t$-compact. So it follows from the equation
$(t+s+i)^{-1}=(t+i)^{-1}-(t+s+i)^{-1}s(t+i)^{-1}$  that $s(t+s+i)^{-1}=s(t+i)^{-1}-s(t+s+i)^{-1}s(t+i)^{-1}$ is a compact operator on $E$. Hence, by Proposition \ref{main3}, one can see that $I+C_{t+s}$ is a compact perturbation of Fredholm operator $I+C_{t}$. Therefore the desired result can be concluded by Proposition \ref{basic} and Theorem \ref{main1}.
\end{proof}

\begin{corollary} \label{compact2}
 Let $t$ be a selfadjoint regular Fredholm operator on a Hilbert $C^*$-module $E$. Suppose that $A$ is a selfadjoint bounded $t$-compact operator on $E$, then $A+t$ is a selfadjoint regular Fredholm operator.
\end{corollary}

\subsection{Perturbations in terms of the gap metric}

The most natural and general definition of smallness of the perturbations of Fredholm operators on Hilbert spaces or Banach spaces is given in terms of the gap metric. Indeed, Theorem  IV4.30 of \cite{Kato} states that the set of Fredholm operators in the space of closed linear operators on a Banach space is open with respect to the gap metric. In this subsection, we generalize this result to regular Fredholm operators.

Let us first briefly recall the definition of the gap metric in the space of regular operators on a Hilbert $C^*$-module and some of its basic properties.

Let $t$ be a regular operator on a Hilbert $C^*$-module $E$  and $G(t)=\{(x,t(x)): x \in D(t)\}$ be its graph. Then Theorem 9.3 of \cite{Lance} implies that $G(t)$ is orthogonally complemented in $E \bigoplus E$, we denote by $ P_{G(t)}$ the orthogonal projection onto $G(t)$. Assume that $t$ and $s$ are regular operators on $E$, define the gap metric between $t$ and $s$ by $g(t,s)=\|P_{G(t)}-P_{G(s)}\|$. 

The matrix representation  of the orthogonal projection onto the graph of a regular operator (see \cite{Lance}) shows that gap metric is equivalent to the metric $d$ given by
\begin{equation} \label{gap}
 d(t,s)=sup\{\|R_{t}-R_{s}\|,\|R_{t^{*}}-R_{s^{*}}\|,\|tR_{t}-sR_{s}\|\}.
 \end{equation}
Suppose $t$ and $s$ are selfadjoint regular operators on a Hilbert $C^*$-module $E$, then Corollary 2.8 of \cite{Gap} implies that 
$$ \frac{1}{4} \|C_{t}-C_{s}\| \leq d(t,s) \leq \frac{1}{2} \|C_{t}-C_{s}\|.$$
  We recall from Proposition 3.7 of \cite{Gap} that $d(t,s)=d(\tilde{t},\tilde{s})$ for all regular operators $t$ and $s$ on $E$, where $\tilde{t}$ and $\tilde{s}$ are defined in Remark \eqref{rm}. %The proof of this result relies heavily on the fact that range and cokernel of a Fredholm operator acting on a Banach spaces or a Hilbert spaces are finite dimensional which does not hold for a regular Fredholm operator acting on a Hilbert $C^*$-module.\\

 %Our criterion for a selfadjoint regular operator to be Fredholm  in Theorem \eqref{main1} allows us to generalize Theorem IV 4.30 of \cite{Kato} to the context of regular operators on Hilbert $C^*$-modules.
\begin{theorem} \label{gap}
Let $t$ be a selfadjoint regular Fredholm operator on a Hilbert $C^*$-module $E$. Then there exists $r >0$ with the following property: If $s$ is a selfadjoint regular operator on $E$ with $g(s,t)\leq r$ then $s$ is Fredholm, where $g$ is the gap metric.
\end{theorem}

\begin{proof}
Note that $I+C_ {t}$ is Fredholm by Theorem \ref{main}, hence we can employ Proposition \ref{basic} to get $\varepsilon  > 0$ such that any bounded adjointable operator $S$ on $E$ satisfying $\|(I+C_{t})-S\|\leq 4 \varepsilon$ is Fredholm. Suppose that $s$ is a selfadjoint regular operator  on $E$ with $d(s,t)\leq \varepsilon$, where $d$ as given in equerry{gap}. Then Corollary 2.8 of \cite{Gap} yields that $\|(I+C_{t})-(I+C_{s})\|\leq 4\varepsilon$. This shows that  $I+C_{s}$ is Fredholm and so $s$ is Fredholm. Finally, observe  that the gap metric is equivalent to $d$, and this finishes the proof.
\end{proof}

Thanks to Remark \ref{rm}, we can deduce the following theorem from Theorem \ref{gap}.
\begin{theorem}\label{gap2}
 Let $t$ be a regular Fredholm operator on a Hilbert $C^*$-module $E$. Then there exists $r >0 $ such that any regular operator $s$ on $E$ satisfying $d(s,t)\leq r$ is also Fredholm.
\end{theorem}
\begin{proof}
Suppose that $t$ is a regular Fredholm operator. By Lemma \ref{sadj}, $\tilde{t}$ is a selfadjoint Fredholm operator. Employ Theorem \ref{gap} to obtain $r > 0$ such that any selfadjoint regular operator satisfying $d(s, \tilde{t}) \leq r$ is Fredholm. We show that this positive scalar $r$ satisfies in the claim of theorem. Let $s$ be regular Fredholm operator satisfying $d(s, t) \leq r$. Then by Proposition 3.7 of \cite{Gap}, $d(\tilde{t}, \tilde{s}) =d(t,s)$, hence $\tilde{s}$ is Fredholm, and so is $s$, as desired.
\end{proof}
\begin{corollary}
 The space of all bounded adjointable Fredholm operators on a Hilbert $C^*$-module $E$ is open in $R(E)$ with respect to the gap metric.
\end{corollary}
\begin{proof}
 Recall that Theorem \ref{gap2} states that the space of regular Fredholm operators is open in $R(E)$ with respect to the gap metric. Moreover, by Theorem 3.4 of \cite{Gap} the space of all bounded adjointable operators on Hilbert $C^*$-module $E$ is open in $R(E)$ with respect to the gap metric. Therefore, the space of bounded adjointable Fredholm operators is open in $R(E)$ with respect to the gap metric, as desired.
 \end{proof}

We conclude this subsection by constructing some continuous paths of selfadjoint regular Fredholm operators with respect to the gap metric. 
\begin{theorem}\label{path1}
Let $t$ be a selfadjoint regular Fredholm operator on a Hilbert $C^*$-module $E$ and let $\lambda\longmapsto A_{\lambda}$ be a continuous path of selfadjoint bounded operators on $E$ with respect to the norm topology. Suppose that each $A_{\lambda}$ is  $t$-compact, then $\lambda \longmapsto t+A_{\lambda}$ is a continuous path of selfadjoint regular Fredholm operators with respect to the gap metric.
\end{theorem}
\begin{proof}
Proposition \ref{compact2} implies that $\lambda \longmapsto t+A_{\lambda}$ is a path of selfadjoint regular Fredholm operators on Hilbert $C^*$-module $E$. By Proposition \ref{main2} we get that
$$\|C_{t+A_{\lambda}}-C_{t+A_{\mu}}\| \leq \|I-C_{t+A_{\mu}}\|\|A_{\lambda}-A_{\mu}\|\|(I+t+A_{\lambda})\|$$
$$\leq 4\|A_{\lambda}-A_{\mu}\|.$$
So Corollary 2.8 of \cite{Gap} shows that $\lambda \longmapsto t+A_{\lambda}$  is a continuous path of selfadjoint regular Fredholm operators with respect to the gap metric.
\end{proof}
\begin{theorem}
Let $t$ be a selfadjoint regular Fredholm operator on a Hilbert $C^*$-module $E$ and let $\varepsilon$ be the scaler as in Theorem \ref{bd1}. Suppose that $\lambda \longmapsto A_{\lambda}$ is a continuous path of selfadjoint bounded operators on $E$ with respect to the norm topology with $\|A_{\lambda}\|\leq \varepsilon$ for each $\lambda$. Then $\lambda \longmapsto t+ A_{\lambda}$ is a continuous path of selfadjoint regular Fredholm operators with respect to the gap metric.
\end{theorem}
\begin{proof}
 It follows from Theorem \ref{bd1} that $\lambda \longmapsto A_{\lambda}+t$ is a path of selfadjoint regular Fredholm operators. By a similar argument given in the proof of Theorem \ref{path1}, one can conclude the continuity of the path with respect to the gap metric
\end{proof}

{\textbf{Acknowledgements}}

 Part of this work was done during author's stay in International School for Advanced Studies (SISSA), Trieste, Italy, 2012. This work was supported by a grant from IPM  and  by SISSA post-graduate fellowship. 
\vskip 0.4 true cm

\bibliographystyle{amsplain}

\end{document}